\documentclass[12pt]{article}
\usepackage{array,amsmath}
\usepackage{amssymb}
\usepackage{graphicx,subfigure}

\makeatletter
 \usepackage{fullpage} 
 \usepackage{amsthm}

\makeatother

\begin{document}

\newtheorem{thm}{Theorem}
\newtheorem{lem}[thm]{Lemma}
\newtheorem{prop}[thm]{Proposition}
\newtheorem{cor}[thm]{Corollary}
\newtheorem{con}[thm]{Construction}
\newtheorem{claim}[thm]{Claim}
\newtheorem{obs}[thm]{Observation}
\newtheorem{defn}[thm]{Definition}
\newtheorem{example}[thm]{Example}
\newcommand{\di}{\displaystyle}
\def\cF{{\cal F}}
\def\cH{{\cal H}}
\def\cK{{\cal K}}
\def\cC{{\cal C}}
\def\cA{{\cal A}}
\def\cB{{\cal B}}
\def\cP{{\cal P}}
\def\ap{\alpha'}
\def\af{\alpha'_f}
\def\dfc{\mathrm{def}}
\def\df{\dfc_f}
\def\Frk{F_k^{2r+1}}
\def\nul{\varnothing} 
\def\st{\colon\,}   
\def\MAP#1#2#3{#1\colon\,#2\to#3}
\def\VEC#1#2#3{#1_{#2},\ldots,#1_{#3}}
\def\VECOP#1#2#3#4{#1_{#2}#4\cdots #4 #1_{#3}}
\def\SE#1#2#3{\sum_{#1=#2}^{#3}}  \def\SGE#1#2{\sum_{#1\ge#2}}
\def\PE#1#2#3{\prod_{#1=#2}^{#3}} \def\PGE#1#2{\prod_{#1\ge#2}}
\def\UE#1#2#3{\bigcup_{#1=#2}^{#3}}
\def\FR#1#2{\frac{#1}{#2}}
\def\FL#1{\left\lfloor{#1}\right\rfloor} 
\def\CL#1{\left\lceil{#1}\right\rceil}  
\def\esub{\subseteq}
\def\SM#1#2{\sum_{#1\in#2}}
\def\NN{{\mathbb N}}

\title{The difference and ratio of the fractional matching number and the matching number of graphs}
\author{
Ilkyoo Choi\thanks{Department of Mathematical Sciences, KAIST, Daejeon, South Korea, \texttt{ilkyoo@kaist.ac.kr}
research partially supported by the National Research Foundation of Korea (NRF) grant funded by the Korea government (MSIP) (NRF-2015R1C1A1A02036398)}
\and
Jaehoon Kim\thanks{School of Mathematics, 
University of Birmingham, Edgbaston, Birmingham, United Kingdom,
\texttt{kimJS@bham.ac.uk} research partially supported by the European Research Council under the European Union's Seventh Framework Programme
(FP/2007-2013) / ERC Grant Agreements no. 306349 (J. Kim)}
\and
Suil O\thanks{Department of Mathematics, Simon Fraser University,
Burnaby, BC, V5A 1S6, \texttt{osuilo@sfu.ca.}
research partially supported by an NSERC Grant of Bojan Mohar}
}

\maketitle

\begin{abstract}
Given a graph $G$, 
the {\it matching number} of $G$, written $\ap(G)$, is the maximum size of a matching in $G$,
and the {\it fractional matching number} of $G$, written $\af(G)$, is the maximum size of a fractional matching of $G$.
In this paper, we prove that if $G$ is an $n$-vertex connected graph that is neither $K_1$ nor $K_3$,
then $\af(G)-\ap(G) \le \frac{n-2}6$ and $\frac{\af(G)}{\ap(G)} \le \FR{3n}{2n+2}$.
Both inequalities are sharp, and we characterize the infinite family of graphs where equalities hold.
 
\end{abstract}

\section {Introduction}
For undefined terms, see~\cite{W}.
Throughout this paper, $n$ will always denote the number of vertices of a given graph.
A {\it matching} in a graph is a set of pairwise disjoint edges.
A {\it perfect matching} in a graph $G$ is a matching in which each vertex has an incident edge in the matching; its size must be $n/2$, where $n=|V(G)|$.
A {\it fractional matching} of $G$ is a function $\phi : E(G) \to [0,1]$ such that for each vertex $v$, $\sum_{e \in \Gamma(v)} \phi(e) \le 1$, where
$\Gamma(v)$ is the set of edges incident to $v$,
and the {\it size of a fractional matching} $\phi$ is $\sum_{e \in E(G)} \phi(e)$.
Given a graph $G$, 
the {\it matching number} of $G$, written $\ap(G)$, is the maximum size of a matching in $G$,
and the {\it fractional matching number} of $G$, written $\af(G)$, is the maximum size of a fractional matching of $G$.

Given a fractional matching $\phi$, since  $\sum_{e \in \Gamma(v)} \phi(e) \le 1$  for each vertex $v$, we have that $2\SM e{E(G)}\phi(e)\le n$, which implies $\af(G)\le n/2$.  
By viewing every matching as a fractional matching it follows that $\af(G) \ge \alpha'(G)$ for every graph $G$, but equality need not hold. 
 For example, the fractional matching number of a $k$-regular graph equals $n/2$ by setting weight $1/k$ on each edge,
 but the matching number of a $k$-regular graph can be much smaller than $n/2$.
Thus it is a natural question to find the largest difference between $\af(G)$ and $\ap(G)$ in a (connected) graph.

In Section 3 and Section 4, we prove tight upper bounds on $\af(G)-\ap(G)$ and $\FR{\af(G)}{\ap(G)}$, respectively, for an $n$-vertex connected graph $G$, and we characterize the infinite family of graphs achieving equality for both results. 
As corollaries of both results, we have upper bounds on both $\af(G)-\ap(G)$ and $\FR{\af(G)}{\ap(G)}$ for an $n$-vertex graph $G$, and we characterize the graphs achieving equality for both bounds.


Our proofs use the famous Berge--Tutte Formula~\cite{B} for the matching number as well as its fractional analogue.
We also use the fact that there is a fractional matching $\phi$ for which $\sum_{e \in E(G)} \phi(e) = \af(G)$
such that $f(e) \in \{0, 1/2, 1\}$ for every edge $e$, and some refinements of the fact.
We can prove both Theorem~\ref{main1} and Theorem~\ref{main2} with two different techniques, and for the sake of the readers we demonstrate each method in the proofs of Theorem~\ref{main1} and Theorem~\ref{main2}.

\section{Tools}
In this section, we introduce the tools we used to prove the main results.
To prove Theorem~\ref{main1}, we use Theorem~\ref{BTF} and Theorem ~\ref{FBTF}.
For a graph $H$, let $o(H)$ denote
the number of components of $H$ with an odd number of vertices.  Given a
graph $G$ and $S\esub V(G)$, define the {\it deficiency} $\dfc(S)$ by
$\dfc(S)=o(G-S)-|S|$, and let $\dfc(G)=\max_{S\esub V(G)}\dfc(S)$.
Theorem~\ref{BTF} is the famous Berge--Tutte formula, which is a general version of Tutte's 1-factor Theorem~\cite{T}.

\begin{thm}[\cite{B}]\label{BTF}
For any $n$-vertex graph $G$, $\ap(G) = \frac 12 \left(n -\dfc{(G)}\right).$
\end{thm}

For the fractional analogue of the Berge--Tutte formula, let $i(H)$ denote the number of isolated vertices in $H$.  
Given a graph $G$ and $S\esub V(G)$, let $\df(S)=i(G-S)-|S|$ and $\df(G)=\max_{S\esub V(G)}\df(S)$.
Theorem~\ref{FBTF} is the fractional version of the Berge--Tutte Formula.
This is also the fractional analogue of Tutte's $1$-Factor Theorem 
saying that $G$ has a fractional perfect matching if and only if $i(G-S)\le|S|$ for all $S\esub V(G)$
(implicit in Pulleyblank~\cite{P}), where a fractional perfect matching is a fractional matching $f$ such that 
$2\sum_{e \in E(G)} f(e)=n$.

\begin{thm}[\cite{SU}~See Theorem 2.2.6]\label{FBTF}
For any $n$-vertex graph $G$, $\af(G)=\frac 12 (n-\dfc_f(G))$.
\end{thm}

When we characterize the equalities in the bounds of Theorem~\ref{main1} and Theorem~\ref{main2},
we need the following proposition.
Recall that $G[S]$ is the graph induced by a subset of the vertex set $S$.

\begin{prop}[\cite{SU}~See Proposition 2.2.2]\label{fracequiv}
The following are equivalent for a graph $G$.\\
(a) $G$ has a fractional perfect matching.\\
(b) There is a partition $\{V_1, \ldots, V_n\}$ of the vertex set $V(G)$ such that, for each $i$, the graph $G[V_i]$ is either $K_2$ or Hamiltonian. \\
(c) There is a partition $\{V_1,\ldots, V_n\}$ of the vertex set $V(G)$ such that, for each $i$, the graph $G[V_i]$ is either $K_2$ or Hamiltonian graph on an odd number of vertices.
\end{prop}

Theorem~\ref{frac01/21} and Observation~\ref{obs} are used to prove Theorem~\ref{main2}.

\begin{thm}[\cite{SU}~See Theorem 2.1.5]\label{frac01/21}
For any graph $G$, there is a fractional matching $f$ for which$$\sum_{e\in E(G)} f(e)=\af(G)$$
such that $f(e) \in \{0, 1/2, 1\}$ for every edge $e$.
\end{thm}

Given a fractional matching $f$, an {\it unweighted} vertex $v$ is a vertex with $\sum_{e \in \Gamma(v)} f(e) = 0$, and a {\it full} vertex $v$ is a vertex with $f(vw)=1$ for some vertex $w$.
Note that $w$ is also a full vertex. An {\it $i$-edge} $e$ is an edge with $f(e)=i$. Note that the existence of an 1-edge guarantees the existence of two full vertices. A vertex subset $S$ of a graph $G$ is {\it independent} if $E(G[S]) = \emptyset$, where $G[S]$ is the graph induced by $S$.

\begin{obs}\label{obs}
Among all the fractional matchings of an $n$-vertex graph $G$ satisfying the conditions of Theorem~\ref{frac01/21}, let $f$ be a fractional matching 
with the greatest number of edges $e$ with $f(e)=1$. Then we have the following:\\
(a) The graph induced by the $\frac 12$-edges is the union of odd cycles. Furthermore, if $C$ and $C'$ are two disjoint cycles in the graph induced by $\frac 12$-edges, then there is no edge $uu'$ such that $u \in V(C)$ and $u'\in V(C')$.\\
(b) The set $S$ of the unweighted vertices  is independent. Furthermore, every unweighted vertex is adjacent only to a full vertex. \\
(c) $\ap(G) \ge w_1 + \sum_{i=1}^{\infty} i c_i$, $\af(G) = w_1 + \sum_{i=1}^{\infty} (\frac{2i+1}2)c_i$, and $n=w_0+2w_1+\sum_{i=1}^{\infty}(2i+1)c_i$, where $w_0$, $w_1$, and $c_i$ are the number of unweighted vertices, the number of 1-edges, and the number of odd cycles of length $2i+1$ in the graph induced by $\frac 12$-edges in $G$, respectively.
\end{obs}
\begin{proof}
(a) The graph induced by the $\frac 12$-edges cannot have a vertex with degree at least 3 since $\sum_{e \in \Gamma(v)} f(e) \le 1$ for each vertex $v$.
Thus the graph must be a disjoint union of paths  or cycles. If the graph contains a path or an even cycle, then by replacing weight 1/2 on each edge on the path or the even cycle with weight 1 and 0 alternatively, we can have a fractional matching with the same fractional matching number and more edges with weight 1, which contradicts the choice of $f$. Thus the graph induced by the $\frac 12$-edges is the union of odd cycles. If there is an edge $uv$ such that $u \in V(C)$ and $v\in V(C')$, where $C$ and $C'$ are two different odd cycles induced by some $\frac12$-edges, then $f(uv)=0$, since $\sum_{e \in \Gamma(x)} f(e) \le 1$ for each vertex $x$. By replacing weights 0 and 1/2 on the edge $uv$ and the edges on $C$ and $C'$ with weight 1 on $uv$, and 0 and 1 on the edges in $E(C)$ and $E(C')$ alternatively, not violating the definition of a fractional matching, we have a fractional matching with the same fractional matching number with more edges with weight 1, which is a contradiction. Thus we have the desired result.\\
(b) If two unweighted vertices $u$ and $v$ are adjacent, then we can put a positive weight on the edge $uv$, which contradicts the choice of $f$. If there exists an unweighted vertex $x$, which is not incident to any full vertex, then $x$ must be adjacent to a vertex $y$ such that $f(yy_1)=1/2$ and $f(yy_2)=1/2$  for some vertices $y_1$ and $y_2$. By replacing the weights 0, 1/2, and 1/2 on $xy$, $yy_1$, and $yy_2$ with 1, 0, 0, respectively,
we have a fractional matching with the same fractional matching number with more edges with weight 1, which is a contradiction. \\
(c) By the definitions of $w_0$, $w_1$, and $c_i$, we have the desired result. 
\end{proof}

\section{Sharp upper bound for $\af(G) - \ap(G)$}\label{lower}
What are the structures of the graphs having the maximum difference between the fractional matching number and the matching number in an $n$-vertex connected graph? 
The graphs may have big fractional matching number and small matching number.
So, by the Berge--Tutte Formula and its fractional version, they may have a vertex subset $S$ such that almost all of the odd components of $G-S$ have at least three vertices in order to get $S$ to have small fractional deficiency and big deficiency. 
This is our idea behind the proof of Theorem~\ref{main1}.

\begin{figure}
\begin{center}
\includegraphics[height=4.5cm]{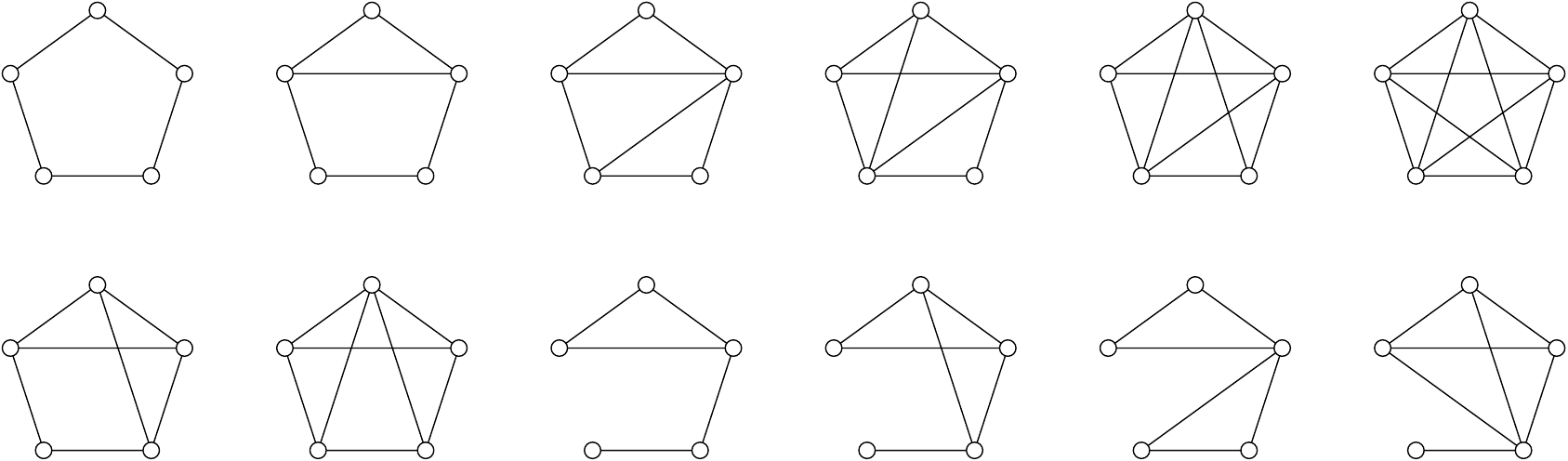}
\end{center}
\caption{All 5-vertex graphs in Theorem~\ref{main1} $(i)$ and Theorem~\ref{main2} $(i)$}
\end{figure}

\begin{figure}
\begin{center}
\includegraphics[height=4.5cm]{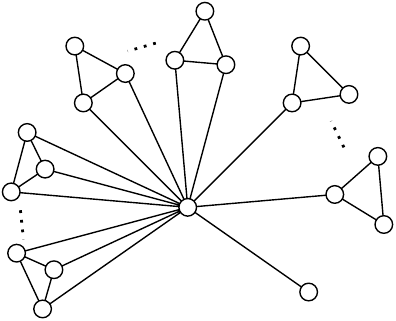}
\end{center}
\caption{All graphs in Theorem~\ref{main1} $(ii)$ and Theorem~\ref{main2} $(ii)$}
\end{figure}

\begin{thm}\label{main1}
For $n \ge 5$, if $G$ is a connected graph with $n$ vertices,
then $\alpha'_f(G) - \alpha'(G) \le \frac{n-2}6$, and equality holds only when either \\
(i) $n=5$ and either $C_5$ is subgraph of $G$ or  $K_2+K_3$ is a subgraph of $G$, or \\
(ii) $G$ has a vertex $v$ such that the components of $G-v$ are all $K_3$ except one single vertex.
\end{thm}

\begin{proof}
Among all the vertex subsets with maximum deficiency, let $S$ be the largest set. 
By the Berge-Tutte Formula, $\alpha'(G)=\frac 12 (n-\dfc(S))$,
and by the choice of $S$, all components of $G-S$ have an odd number of vertices.
Let $x$ be the number of isolated vertices of $G-S$, and let $y$ be the number of other components of $G-S$.
This implies $n \ge |S|+x+3y$. If $S=\emptyset$, then $\alpha'(G)\in\{\frac n2, \frac{n-1}2\}$, depending on the parity of $n$. 
In this case, $\alpha'_f(G) - \alpha'(G) \le \frac n2 - \frac{n-1}2 = \frac 12 \le \frac{n-2}6$, since $n \ge 5$.
Now, assume that $S$ is non-empty.\\
{\it Case 1: $x=0$}.
Since $\dfc_f(G) \ge 0$, $|S| \ge 1$, and $n \ge |S|+3y$, we have 
$$\alpha'_f(G)-\alpha'(G) = \frac 12 (n-\dfc_f(G)) - \frac 12 (n-\dfc(S)) = \frac 12 (\dfc(S) - \dfc_f(G))$$ 
$$\le \frac 12 (y - |S|  - 0) \le \frac 12(\frac{n-|S|}3 -|S|) = \frac {n-4|S|}6 \le \frac{n-4}6 < \frac {n-2}6.$$
{\it Case 2: $x \ge 1$}. Since $n \ge |S|+x+3y$, $|S|\ge1$, and $x \ge 1$, we have
$$\alpha'_f(G)-\alpha'(G) = \frac 12 (n-\dfc_f(G)) - \frac 12 (n-\dfc(S)) = \frac 12 (\dfc(S) - \dfc_f(G))$$ 
$$\le \frac 12 \left( x+y - |S|  - (x-|S|) \right) \le \frac y2 = \frac {n-x-|S|}6 \le \frac{n-2}6.$$
Equality in the bound requires equality in each step of the computation.
When $n=5$, we conclude that $(i)$ follows by Proposition~\ref{fracequiv}.
In Case 1, we cannot have equality, and in Case 2, we have $|S|=1$, $x=1$, and $n=|S|+x+3y=2+3y$. 
Since $G$ is connected, the components of $G-S$ are $P_3$ or $K_3$ except only one single vertex.
If a component of $G-S$ is a copy of $P_3$, then by choosing the central vertex $u$ of the path, 
we have $\dfc(S\cup \{u\}) = o(G-(S\cup \{u\}))-|S\cup\{u\}|=o(G-S)-|S|$, yet $|S\cup\{u\}| > |S|$,
which contradict the choice of $S$. Thus we have the desired result.
\end{proof}

\begin{cor} \label{main1cor}
For any $n$-vertex graph $G$, we have $\alpha'_f(G) - \alpha'(G) \le \frac n6$, and equality holds only when $G$ is the disjoint union of copies of $K_3$.
\end{cor}
\begin{proof}
First, we show that if $n \le 4$ and $G$ is connected, then $\af(G)-\ap(G) \le \frac n6$, and equality holds only when $G=K_3$.
If $n \le 2$, then $G\in\{K_1, K_2\}$, which implies that $\af(G)-\ap(G) = 0 < n/6$. If $n=3$, then $G\in\{P_3, K_3\}$.
Note that $\af(P_3)-\ap(P_3)=1-1=0 < 3/6$ and $\af(K_3)-\ap(K_3)=3/2 - 1 = 1/2 \le 3/6$.
Furthermore, equality holds only when $G=K_3$. 
If $n=4$, then either $G=K_{1,3}$ or $G$ contains $P_4$ as a subgraph.
Since $\af(K_{1,3}) - \ap(K_{1,3})=1-1=0 < 4/6$ and $\af(P_4)-\ap(P_4)=2-2=0 < 4/6$,
we conclude that for any positive integer $n$, $\af(G)-\ap(G) \le \frac n6$. In fact, if $n \ge 5$,  then by Theorem~\ref{main1}, the difference must be at most $\frac {n-2}6$. Thus, for connected graphs, equality holds only when $G=K_3$ .

Now, if we assume that $G$ is disconnected, then $G$ is the disjoint union of connected graphs $G_1, \ldots, G_k$. Let $|V(G_i)|=n_i$ for $i \in [k]$.
Since
$$\af(G)-\ap(G) = \left[\af(G_1)+\cdots +\af(G_k)\right] - \left[\ap(G_1)+\cdots +\ap(G_k)\right]$$ $$= [\af(G_1)-\ap(G_1)] + \cdots +  [\af(G_k)-\ap(G_k)] \le \frac{n_1}6 + \cdots + \frac{n_k}6 = \frac n6,$$ equality holds only when each $G_i$ is a copy of $K_3$ for $i \in [k]$.
\end{proof}

\section{Sharp upper bound for $\frac{\af(G)}{\ap(G)}$}\label{lower}
To prove the upper bound of Theorem~\ref{main2}, we still can use the Berge-Tutte formula and its fractional analogue.
However, we provide an alternative way to prove the theorem.

\begin{thm}\label{main2}
For $n \ge 5$, if $G$ is a connected graph with $n$ vertices, then $\frac{\af(G)}{\ap(G)} \le \frac{3n}{2n+2}$,
and equality holds only when either\\
(i) $n=5$ and either $C_5$ is a subgraph of $G$ or  $K_2+K_3$ is a subgraph of $G$, or \\
(ii) $G$ has a vertex $v$ such that the components of $G-v$ are all $K_3$ except one single vertex.
\end{thm}

\begin{proof}
Among all the fractional matchings of an $n$-vertex graph $G$ with the size equal to $\af(G)$, let $f$ be a fractional matching such that the number of edges $e$ with $f(e)=1$ is maximized. We follow the notation in Observation~\ref{obs}.\\
{\it Case 1: $w_0=w_1=0$.} Since $G$ is connected and $n\ge 5$, there exists only one $i$ such that $i \ge 2$ and $c_i$ is not zero, and $\ap(G)=ic_i\neq 0$.
Then we have 
$$\frac{\af(G)}{\ap(G)} \le \frac{(\frac{2i+1}{2})c_i}{ic_i} = 1 + \frac{1}{2i} \le \frac 54.$$
{\it Case 2: $w_0 \ge 1$ and $w_1=0$}. By part (b) of Observation~\ref{obs}, this cannot happen.\\
{\it Case 3: $w_0=0$ and $w_1 \ge 1$.}  Since $\sum_{i=1}^{\infty}c_i \le \frac{n-2w_1}3$, by part (c) of Observation~\ref{obs}, we have
$$\frac{\af(G)}{\ap(G)} \le \frac{w_1+\sum_{i=1}^{\infty}(\frac{2i+1}2)c_i}{w_1+\sum_{i=1}^{\infty}ic_i}
=\frac{\frac{n-w_0}2}{\frac{n-w_0-\sum_{i=1}^{\infty}c_i}2}=\frac{n}{n-\sum_{i=1}^{\infty}c_i} \le \frac n{n-\frac{n-2w_1}3}=\frac{3n}{2n+2w_1}\le \frac {3n}{2n+2}.$$
{\it Case 4: $w_0\! \ge \!1$ and $w_1 \!\ge \!1$.} Since $\sum_{i=1}^{\infty}c_i \le \frac{n-2w_1-w_0}3$, by part (c) of Observation~\ref{obs}, we have
$$\frac{\af(G)}{\ap(G)} \le  \frac{\frac{n-w_0}2}{\frac{n-w_0-\sum_{i=1}^{\infty}c_i}2}  \le 
\frac{n-w_0}{n-w_0-\frac{n-2w_1-w_0}3}  = \frac{3(n-w_0)}{2(n+w_1-w_0)} < \frac{3n}{2(n+w_1)} \le \frac{3n}{2(n+1)}.$$
Equality in the bound requires equality in each step of the computation; we only need to check Case 1 and Case 2.
In Case 1, we have $i=2$, which means that $n=5$ and $G$ contains a copy of $C_5$. 
In Case 3, we have $w_1=1$ and $\sum_{i=1}^{\infty}c_i=\frac{n-2}3$, which means that the graph induced by the $\frac 12$-edges is the union of $K_3$.
Thus $G$ has $K_2+kK_3$ as a subgraph for some positive integer $k$. 
Note that there is an edge between the copy of $K_2$ and any copy of $K_3$ by part (b) of Observation~\ref{obs}. Also, there are no edges between any pair of two trianges by part (a) of Observation~\ref{obs}. 
Let $u$ and $v$ be the two vertices corresponding to the copy of $K_2$.
If there are two different triangles $C$ and $C'$ in $G$ such that $u$ and $v$ are incident to $C$ and $C'$, respectively, 
then we have $\ap(G) > w_1 + c_1$, which implies that we cannot have equality in the first inequality in Case 3.
Thus, we conclude that $G$ contains a copy of either $K_2+K_3$ as a subgraph or a vertex $v$ such that the components of $G-v$ are all $K_3$ except only one single vertex.
\end{proof}

\begin{cor}
For any $n$-vertex graph $G$ with at least one edge, we have $\frac{\af(G)}{\ap(G)} \le \frac 32$, and equality holds only when $G$ is the disjoint union of copies of $K_3$.
\end{cor}
\begin{proof}
By the proof of Corollary~\ref{main1cor}, if $n \le 4$ and $G$ is connected, then $\frac{\af(G)}{\ap(G)} \le \frac 32$, and equality holds only when $G=K_3$.
If we assume that $G$ is disconnected, then $G$ is the disjoint union of connected graphs $G_1, \ldots, G_k$. Let $|V(G_i)|=n_i$ for $i \in [k]$.
Without loss of generality, we may assume that $\frac{\af(G_1)}{\ap(G_1)} \ge \frac{\af(G_i)}{\ap(G_i)}$ for all $i \in [k]$. Then we have
$$\frac{\af(G)}{\ap(G)} = \frac{\af(G_1)+\cdots +\af(G_k)}{\ap(G_1)+\cdots +\ap(G_k)}\le \frac{\af(G_1)}{\ap(G_1)} \le \frac 32,$$ and equality holds only when each $G_i$ is a copy of $K_3$ for $i \in [k]$.
\end{proof}

\end{document}